\documentclass[10pt,reqno]{amsart}

\usepackage{a4wide}
\usepackage{dsfont} 
\usepackage{pstricks}
\usepackage[all]{xy}  
\usepackage{amssymb,amsmath, comment}
\usepackage{mathrsfs}
\usepackage{graphicx}

\newtheorem{proposition}{Proposition}[section]
\newtheorem{theorem}[proposition]{Theorem}
\newtheorem{lemma}[proposition]{Lemma}

\usepackage{hyperref}

\theoremstyle{remark}
\newtheorem{definition}[proposition]{Definition}

\newtheorem{remark}{Remark}
\newcommand{\I}{\mathds{1}}

\newcommand{\id}{\mathrm{id}}
\newcommand{\HH}{\mathbb{H}}
\newcommand{\GG}{\mathbb{G}}

\newcommand{\AM}{\mathcal{A}}

\newcommand{\BM}{\mathcal{B}}
\newcommand{\CM}{\mathcal{C}}
\newcommand{\C}{\mathrm{C}}
\newcommand{\M}{\mathrm{M}}
\newcommand{\B}{\mathrm{B}}
\newcommand{\Mor}{\mathrm{Mor}}

\newcommand{\ww}{\mathrm{W}}
\newcommand{\WW}{{\mathds{V}\!\!\text{\reflectbox{$\mathds{V}$}}}}

\newcommand{\Ob}{\mathrm{Obj}}
\newcommand{\mC}{\mathcal{C}^*}

\newcommand{\cad}{ad}
\numberwithin{equation}{section}

\begin{document}
\author{Mariusz Budzi{\'n}ski}
\address{Faculty of Physics, University of Warsaw, Poland}
\email{mb348587@okwf.fuw.edu.pl }
\author{Pawe{\l} Kasprzak}
\address{Department of Mathematical Methods in Physics, Faculty of Physics, University of Warsaw, Poland}
\email{pawel.kasprzak@fuw.edu.pl}
\title{Quantum families  of   quantum group  homomorphisms}

\thanks{This work was supported by the "Harmonia" NCN grant
2012/06/M/ST1/00169.}

\begin{abstract} 
The notion of a quantum family of maps  has been introduced in the  framework of $\C^*$-algebras. As  in the classical case, one may consider  a  quantum family of maps preserving additional structures (e.g.  quantum family of maps preserving a  state).  In this paper we define a quantum family of homomorphisms of locally compact quantum groups. Roughly speaking, we  show that   such a family  is  classical.   The purely algebraic counterpart of the discussed notion, i.e. a quantum family of homomorphisms of Hopf algebras, is introduced  and   the algebraic counterpart of the aforementioned  result is proved.  Moreover, we show that a quantum family of homomorphisms of Hopf algebras is  consistent with the counits and coinverses of the given Hopf algebras. We compare our concept with {\it weak coactions} introduced by Andruskiewitsch and  we apply it  to the analysis of adjoint coaction. 
\end{abstract}
\maketitle
 
\section{Introduction}
Gelfand-Najmark theorem identifies  a commutative $\C^*$-algebra $A$ with the algebra of  continuous, vanishing at infinity   functions on the spectrum $\textrm{Sp}(A)$. This identification  is a  source of many concepts in the theory of non-commutative $\C^*$-algebras, e.g.   quantum family of maps between quantum spaces. The notion was introduced in  \cite{Wo1} and further developed in   \cite{solfam}, \cite{Wang}.  For the recent survey article on this subject we refer to \cite{SkSol}. Since  we  also consider  the purely algebraic counterpart of the concept, we prefer to use the name  
 \textit{quantum family of morphisms}  instead of quantum family of maps between quantum spaces. 
 
Given   $\C^*$-algebras $A_1,A_2$ and $B$, a quantum family of morphisms from $A_1$ to $A_2$  is a morphism $\alpha\in\Mor(A_1,B\otimes A_2)$. For $B =\C_0(X)$, $\alpha$ yields a continuous family of morphisms indexed by the elements of $X$. Indeed,  denoting the character assigned to $x\in X$ by $\textrm{ev}_x:\C_0(X)\rightarrow \mathbb{C}$, we may view $\alpha$ as a continuous map 
\[X\ni x\mapsto \alpha_x\in\Mor(A_1,A_2)\] where 
\[\alpha_x = (\textrm{ev}_x\otimes\id_{A_2})\circ\alpha.\]
 For the description of topology on $\Mor(A_1,A_2)$ we refer to \cite{WorGen}.
 
Given  quantum family   of morphisms may   be consistent with  certain     structures on $A_1$ and $A_2$. For instance if  $A_1=A_2=A$ then      $\alpha\in\M(A,B\otimes A)$ is said to preserve a state $\omega:A\rightarrow\mathbb{C}$ if
\[(\id\otimes\omega )\circ\alpha(a) =\omega (a) \I_{B}.\]
In this paper we consider the case when $A_1$ and $A_2$ are $\C^*$-algebras assigned to locally compact quantum groups $\GG$ and $\HH$ and  we define a quantum family of homomorphisms from $\HH$ to $\GG$ as quantum family of morphisms satisfying further conditions. 

In order to explain our concept, let us consider a classical   family of homomorphisms from $\HH$ to $\GG$, i.e.  $\alpha\in\Mor(\C_0^u(\GG),\C_0(X)\otimes\C_0^u(\HH))$
such that
\begin{equation}\label{defal}(\alpha_x\otimes\alpha_x)\circ\Delta^u_{\GG} = \Delta^u_{\HH}\circ\alpha_x\end{equation} where $\C_0^u(\HH),\C_0^u(\GG)$ denotes  the universal $\C^*$-algebras assigned to $\HH,\GG$ respectively. Let  $\WW^G\in\M(\C_0^u(\hat\GG)\otimes\C_0^u(\GG))$ be the universal bicharacter of $\GG$ and let us consider  a unitary 
\begin{equation}\label{defU}U = (\id\otimes\alpha)(\WW^\GG)\in\M(\C^u_0(\hat\GG)\otimes \C_0(X)\otimes\C^u_0(\HH)).\end{equation} 
Then, using the {\it leg numbering notation},   \eqref{defal} is equivalent with the following identity for $U$
\begin{equation}\label{mcond}(\id_{\C_0^u(\hat\GG)}\otimes\id_{\C_0(X)}\otimes\Delta^u_{\HH})(U) = U_{123}U_{124}.\end{equation} 
This inspires us to  define  a quantum family of homomorphisms of quantum groups, as quantum family of morphisms 
\[\alpha\in\Mor(\C_0^u(\GG),B\otimes\C_0^u(\HH))\] such that, defining $U = (\id\otimes\alpha)(\WW^\GG)\in\M(\C_0^u(\hat\GG)\otimes B\otimes\C_0^u(\HH))$ we have
\[(\id_{\C_0^u(\hat\GG)}\otimes \id_B\otimes\Delta_{\C_0^u(\HH)})(U) = U_{123}U_{124}\]
where we used the  the leg numbering notation. 
We prove that    a quantum family $\alpha$ of homomorphisms of quantum groups satisfies 
\[\alpha(x)_{12} \alpha(y)_{13} =\alpha(y)_{13} \alpha(x)_{12}\] for all $x,y\in\C_0^u(\GG)$ (see Theorem \ref{thm0}). A special   case $\HH = \GG$ of this result   was already obtained in \cite{QAGC} and it led   to a proof  that a quantum group of automorphisms of a finite quantum group as defined in \cite{QAG} is  its classical group of automorphisms.

In the second part of the paper we move on to a purely algebraic context of the category \emph{Alg} of unital  algebras over a field $k$.  Let   $\mathcal{H}_1 = (\AM_1,\Delta_{\AM_1},S_{\AM_1}, \varepsilon_{\AM_1})$, $\mathcal{H}_2 = (\AM_2,\Delta_{\AM_2}, S_{\AM_2},  \varepsilon_{\AM_2})$ be  Hopf algebras and $\BM$ an algebra. The multiplication map on $\BM$ will be denoted $m_\BM:\BM\otimes\BM\rightarrow \BM$.  In order to formulate the definition of  a quantum family   $\alpha:\AM_1\rightarrow \BM\otimes\AM_2$  of homomorphisms, 
  let us consider an $n$-tuple $(\alpha_1,\alpha_2,\ldots,\alpha_n)$ of Hopf algebras homomorphisms $\alpha_i:\AM_1\rightarrow  \AM_2$  where
\begin{equation}\label{condhom} \Delta_{\AM_2}\circ \alpha_i = (\alpha_i\otimes\alpha_i)\circ \Delta_{\AM_1}.\end{equation}
 Denoting the commutative algebra $k^n$ by $\BM$ and  identifying $\BM\otimes\AM_2\cong\AM_2^n$ we   may view the $n$-tuple $(\alpha_1(a),\alpha_2(a),\ldots,\alpha_n(a))$ as   
$\alpha:\AM_1\rightarrow\BM\otimes\AM_2$
\[\alpha(a) = (\alpha_1(a),\alpha_2(a),\ldots,\alpha_n(a)).\] The condition \eqref{condhom} reads
\begin{equation}\label{defining} (m_{\BM}\otimes\id\otimes\id)\circ(\id\otimes\sigma_{\AM_2\otimes\BM}\otimes\id)\circ(\alpha\otimes\alpha)\circ\Delta_{\AM_1}=(\id\otimes\Delta_{\AM_2})\circ\alpha\end{equation} where 
\[\sigma_{\AM_2,\BM}:\AM_2\otimes\BM\rightarrow\BM\otimes\AM_2\] is the flip homomorphism
\[\sigma_{\AM_2,\BM}(a\otimes b) = b\otimes a.\]  
Allowing $\BM$ to be an arbitrary algebra over $k$, we define a quantum family of Hopf algebra homomorphisms to be a unital homomorphism $\alpha:\AM_1\rightarrow \BM\otimes\AM_2$  satisfying \eqref{defining}.
Let us note that for  $\BM=k^n$   we  have 
\begin{equation}\label{eqcomm} \alpha(x)_{12} \alpha(y)_{13} =\alpha(y)_{13} \alpha(x)_{12} \end{equation} for all $x,y\in\AM_1$  
and 
 \begin{equation}\label{eqcons}
 \begin{split}
 (\id\otimes S_{\AM_2})(\alpha(a)) &= \alpha(S_{\AM_1}(a))\\
 (\id\otimes \varepsilon_{\AM_2})(\alpha(a)) &= \varepsilon_{\AM_1}(a)\I_{\BM}
 \end{split}
 \end{equation} for all $a\in\AM_1$. 
 Indeed, \eqref{eqcomm} follows from the commutativity of $k^n$ while \eqref{eqcons} is equivalent with 
 \[ \begin{split}
  S_{\AM_2}(\alpha_i(a)) &= \alpha_i(S_{\AM_1}(a))\\
 \varepsilon_{\AM_2}(\alpha_i(a)) &= \varepsilon_{\AM_1}(a)
 \end{split}\]
 for all $i\in\{1,2,\ldots,n\}$. 
We prove   that  \eqref{eqcomm} and \eqref{eqcons} hold for $\alpha$ satisfying \eqref{defining} with arbitrary $\BM$ (see Theorem \ref{mainthm}, Theorem \ref{counitqfh} and Theorem \ref{Spres}). 
Thus we answer  positively the question if   \eqref{eqcons} follows from \eqref{defining} which appeared in \cite{QAG}.

 The concept of quantum family of homomorphisms is closely related with Andruskiewitsch's {\it weak coaction} (see \cite{AD1}) which is explained in Remark \ref{weakcorem}. Since the latter was used for the  discussion of  adjoint coaction our results has some consequences in this context (see Section  \ref{exac}). As a side considerations  we introduce  the concept of cocentralizer of a given algebra homomorphism $\Phi:\AM\to\BM$ and we characterize it in terms of the adjoint coaction. 

The  results of this paper   have   certain similarities with the results obtained in \cite{Gos}. The latter concerns   a  Hopf algebra $\mathcal{Q}$ coacting on a unital commutative algebra $\AM$. The authors formulate some conditions    (phrased in terms of a bilinear form which is preserved under the coaction) which forces  the Hopf algebra $\mathcal{Q}$ to be commutative. Viewing the coaction as a quantum family of maps of $\AM$ the analogy of this result with the context of our paper is clear. 

\section{A quantum family of homomorphisms of locally compact quantum groups}
Let  $\mathcal{C}^*$  be the category of $\C^*$-algebras. The objects of $\mathcal{C}^*$ are $\C^*$-algebras 
and a morphism $\pi\in\Mor(A,B)$ is   a non-degenerate $*$-homomorphism: $\pi:A\rightarrow\M(B)$, $\pi(A)B = B$. The tensor product $A\otimes B$ is the  spatial one.  For the discussion of morphisms and tensor product  in the context of  $\mathcal{C}^*$-category  we refer to \cite{WorGen}. 

\begin{definition}
Let $A_1,A_2,B\in\Ob(\mC)$. A morphism $\alpha\in\Mor(A_1,B\otimes A_2)$ will be called a {\it quantum family of morphisms} from $A_1$ to $A_2$. 
\end{definition}

In what follows we shall introduce  and use certain elements of the theory of locally compact quantum groups (lcqg). 
For the axiomatic formulations of lcqg with the existence of Haar weights postulated  we refer  to \cite{KV} or \cite{mnw}. For the theory  with the multiplicative unitary playing the central role we refer to \cite{SolWor}. For the needs of   this paper the theory based on  multiplicative unitary  is sufficient. We shall freely use the leg numbering notation. 

Let $\GG$ be a locally compact quantum group and  $\hat\GG$ its dual. The  multiplicative unitary $\ww^\GG$ of $\GG$ may be viewed as  (the reduced)  bicharacter $\ww^\GG\in \M(\C_0(\hat\GG)\otimes\C_0(\GG))$.  
\[
\begin{split}
(\id\otimes\Delta_\GG)(\ww^\GG) &= \ww^\GG_{12}\ww^\GG_{13}\\
( \Delta_{\hat\GG}\otimes\id)(\ww^\GG) &= \ww^\GG_{23}\ww^\GG_{13}.
\end{split}
\]
Let $B$ be a $\C^*$-algebra and $\pi_1,\pi_2\in\Mor(\C_0(\GG),B)$.
If
\[(\id\otimes\pi_1)(\ww^\GG) = (\id\otimes\pi_2)(\ww^\GG)\] then $\pi_1 = \pi_2$ which easily follows from 
\[\C_0(\GG) =  \{(\omega\otimes\id)(\ww^\GG):\omega\in\B(L^2(\GG))_*\}^{-\|\cdot\|}.\] 

 One assigns with $\GG$ and $\hat\GG$ their universal versions. For the constructions of the universal version within the   multiplicative unitary framework  see \cite{SolWor}.   Thus one constructs $\C^u_0(\GG), \C^u_0(\hat\GG)$, the reducing morphisms     $\Lambda_\GG\in\Mor(\C_0^u(\GG),\C_0(\GG))$, $\Lambda_{\hat\GG}\in\Mor(\C_0^u(\hat\GG),\C_0(\hat\GG))$ and comultiplications   \[\begin{split}\Delta^u_{\GG}&\in\Mor(\C_0^u(\GG),\C_0^u(\GG)\otimes\C_0^u(\GG))\\\Delta^u_{\hat\GG}&\in\Mor(\C_0^u(\hat\GG),\C_0^u(\hat\GG)\otimes\C_0^u(\hat\GG)).\end{split}\] It turns out that $\ww^\GG$ lifts to the universal bicharacter $\WW^\GG\in\M(\C^u_0(\hat\GG)\otimes\C^u_0(\GG))$  (see e.g. \cite{SLW12}), i.e. 
 \begin{equation}\label{biu}\begin{split}(\id\otimes\Delta_{\GG}^u)(\WW^\GG) &= \WW^\GG_{12}\WW^\GG_{13}\\
 (\Delta_{\hat\GG}^u\otimes\id)(\WW^\GG) &= \WW^\GG_{23}\WW^\GG_{13}\end{split}\end{equation}  and
 \[\ww^\GG = (\Lambda_{\hat\GG}\otimes\Lambda_\GG)(\WW^\GG).\] 
 As in the reduced case, if   $\pi_1,\pi_2\in\Mor(\C^u_0(\GG),B)$ satisfies
 \begin{equation}\label{pi}(\id\otimes\pi_1)(\WW^\GG) = (\id\otimes\pi_2)(\WW^\GG)\end{equation} then $\pi_1=\pi_2$ which is the consequence of the equality
 \begin{equation}\label{gen}
 \C_0^u(\GG) = \{(\omega\circ\Lambda_{\hat\GG}\otimes\id_{\C_0^u(\GG)})(\WW^\GG):\omega\in\B(L^2(\GG))_*\}^{-\|\cdot\|}.
 \end{equation}

Let $\HH$ and $\GG$ be locally compact quantum groups and let  $\alpha\in\Mor(\C_0^u(\GG),\C_0^u(\HH))$. We say that $\alpha$ is a homomorphism from $\HH$ to $\GG$ if 
\[(\alpha\otimes\alpha)\circ\Delta^u_{\GG} = \Delta^u_{\HH}\circ\alpha.\] 
A morphism  $\alpha\in\Mor(\C_0^u(\GG),\C_0^u(\HH))$ may be assigned  with a unitary 
\[U = (\id_{\C_0^u(\hat\GG)}\otimes\alpha)(\WW^\GG)\in\M(\C_0^u(\hat\GG)\otimes\C_0^u(\HH)).\]
Clearly,  $\alpha$ is a homomorphism from $\HH$ to $\GG$ if and only if 
\[(\id\otimes\Delta_{\HH}^u)(U) = U_{12}U_{13}.\]
 \begin{definition}
 Let $B$ be a $\C^*$-algebra and  $\HH$, $\GG$   locally compact quantum groups. Let   $\alpha\in\Mor(\C_0^u(\GG),B\otimes\C_0^u(\HH))$ and let us  define  $U = (\id_{\C^u_0(\hat\GG)}\otimes\alpha)(\WW^\GG)\in\M(\C^u_0(\hat\GG)\otimes B\otimes\C_0^u(\HH))$.  We say that $\alpha$  is a {\it quantum family of homomorphisms} from $\HH$ to $\GG$ if   $U$ satisfies 
 \begin{equation}\label{defU1}(\id_{\C^u_0(\hat\GG}\otimes\id_B\otimes\Delta^u_{\HH})(U) = U_{123}U_{124}.\end{equation}
\end{definition}
Let us note that the second equation of  \eqref{biu} yields
\begin{equation}\label{propU}(\Delta^u_{ \hat\GG}\otimes\id_B\otimes\id_{\C^u_0(\HH)})(U) = U_{234}U_{134}.\end{equation}

\begin{theorem}\label{thm0}
 Let $\alpha\in\Mor(\C_0^u(\GG),B\otimes\C_0^u(\HH))$ be a quantum family of homomorphisms from $\HH$ to $\GG$.  Then 
 \[\alpha_{12}(x)\alpha_{13}(y) = \alpha_{13}(y)\alpha_{12}(x)\] for all $x,y\in\C_0^u(\GG)$.
\end{theorem}
\begin{proof}
The morphism \[\Delta_{\hat\GG}^u\otimes\id_B\otimes\Delta_{\HH}^u\in\Mor(\C_0^u(\hat\GG)\otimes B\otimes\C_0^u(\HH),\C_0^u(\hat\GG)\otimes\C_0^u(\hat\GG)\otimes B\otimes\C_0^u(\HH)\otimes\C_0^u(\HH) )\] may be written  in two ways 
\[
\begin{split}
\Delta_{\hat\GG}^u\otimes\id_B\otimes\Delta_{\HH}^u & = (\Delta_{\hat\GG}^u\otimes\id_B\otimes\id_{\C_0^u(\HH)}\otimes\id_{\C_0^u(\HH)}) \circ(\id_{\C_0^u(\hat\GG)}\otimes\id_B\otimes\Delta^u_{\HH})\\
\Delta_{\hat\GG}^u\otimes\id\otimes\Delta_{\HH}^u & =(\id_{\C_0^u(\hat\GG)}\otimes\id_{\C_0^u(\hat\GG)}\otimes\id_B\otimes\Delta^u_{\HH})\circ (\Delta_{\hat\GG}^u\otimes\id_B\otimes\id_{\C_0^u(\HH)}).
\end{split}
\]  Applying these two forms of $\Delta_{\hat\GG}^u\otimes\id_B\otimes\Delta^u_{\HH}$ to   $U\in\M(\C^u_0(\hat\GG)\otimes B\otimes\C_0^u(\HH)) $  and using \eqref{defU1} and \eqref{propU} we get 
\[(\Delta^u_{\hat\GG}\otimes\id_B\otimes \Delta_{\HH}^u)(U) = U_{234}U_{134}U_{235}U_{135}\] 
and 
\[(\Delta^u_{\hat\GG}\otimes\id_B\otimes\Delta_{\HH}^u)(U) =  U_{234}U_{235}U_{134}U_{135}.\]
In particular

\begin{equation}\label{coU}U_{134}U_{235} = U_{235}U_{134}\end{equation} Let $\mu,\nu\in \C_0^u(\hat\GG)^*$. Applying   $\mu\otimes\nu \otimes\id_B\otimes\id_{\C_0^u(\HH)}\otimes\id_{\C_0^u(\HH)}$ to  \eqref{coU} and using  \eqref{gen} we get  
\[\alpha(x)_{12}\alpha(y)_{13} = \alpha(y)_{13}\alpha(x)_{12}\]
for all $x,y\in\C_0^u(\GG)$. 
\end{proof}

\subsection{Adjoint coaction}
Let $\GG$ be a locally compact quantum group and let us denote the universal bicharacter of $\GG$ by   $\WW^\GG\in\M(\C_0^u(\hat\GG)\otimes\C_0^u(\GG))$. Let us consider 
a quantum family of maps $\alpha\in\Mor(\C_0^u(\GG),\C_0^u(\hat\GG)\otimes\C_0^u(\GG))$ given by 
\[\alpha(x) = \WW^\GG(\I_{\C_0^u(\hat\GG)}\otimes x)\WW^{\GG*}\] for all $x\in\C_0^u(\GG)$. 
\begin{theorem}\label{examp}
A quantum family of maps $\alpha\in\Mor(\C_0^u(\GG),\C_0^u(\hat\GG)\otimes\C_0^u(\GG))$ is a quantum family of homomorphisms from $\GG$ to $\GG$ if and only if for all $x\in\C_0^u(\GG)$ and $y\in\C_0^u(\hat\GG)$ we have 
\[\alpha(x)(y\otimes\I) = (y\otimes\I)\alpha(x).\]
\end{theorem}
\begin{proof}
Noting that $U =  (\id_{\C_0^u(\hat\GG)}\otimes\alpha)(\WW^\GG) = \WW^\GG_{23}\WW^\GG_{13}\WW_{23}^{\GG*}$ we compute 
\[\begin{split}
 (\id_{\C_0^u(\hat\GG)}\otimes\id_{\C_0^u(\hat\GG)}\otimes\Delta_{\GG}^u)(U) &= \WW^\GG_{23}\WW^\GG_{24}\WW^\GG_{13}\WW^\GG_{14}\WW_{24}^{\GG*}\WW_{23}^{\GG*}\\
 &=\WW^\GG_{23}\WW^\GG_{13}\WW^\GG_{24}\WW^\GG_{14}\WW_{24}^{\GG*}\WW_{23}^{\GG*}.
 \end{split}
\]
On the other hand 
\[U_{123}U_{124} = \WW^{\GG}_{23}\WW^{\GG}_{13}\WW_{23}^{\GG*}\WW^{\GG}_{24}\WW^{\GG}_{14}\WW_{24}^{\GG*}.\]
The equality $ (\id_{\C_0^u(\hat\GG)}\otimes\id_{\C_0^u(\hat\GG)}\otimes\Delta_{\GG}^u)(U) = U_{123}U_{124}$ yields
\[\WW^{\GG}_{24}\WW^{\GG}_{14}\WW_{24}^{\GG*}\WW_{23}^{\GG*} = \WW_{23}^{\GG*}\WW_{24}^{\GG}\WW_{14}^{\GG}\WW_{24}^{\GG*}\]
which we may write as 
\[\WW_{23}^{\GG}(\id\otimes\alpha)(\WW^{\GG})_{124} =(\id\otimes\alpha)(\WW^{\GG})_{124} \WW^{\GG}_{23}.\]
Reasoning as  at the end of the proof of Theorem \ref{thm0}, we get the desired result. 
\end{proof}
Let us note that 
\begin{itemize}
\item if $\C_0^u(\GG)$ is commutative then the action $\alpha$ is trivial and the condition of Theorem \ref{examp} holds;
\item  if $\C_0^u( \GG)$ is cocommutative (thus $\C_0^u(\hat\GG)$ is   commutative) then the condition of Theorem \ref{examp}  holds. 
\end{itemize}
 \section{Hopf algebraic context}\label{HR} 
Let $k$ be a field and \emph{Alg}  the category of unital algebras over $k$. The tensor product of $\BM,\CM\in\Ob (\emph{Alg})$ is denoted $\BM\otimes\CM$ and the flip morphism by
$\sigma_{\BM,\CM}:\BM\otimes\CM\rightarrow\CM\otimes\BM$.  The multiplication map of $\BM$ is denoted $m_\BM:\BM\otimes\BM\rightarrow\BM$; $m_{\BM}$   is  a morphism if and only if $\BM$ is commutative.

A Hopf algebra $\mathcal{H}$ is a quadruple $\mathcal{H} = (\AM,\Delta_{\AM},\varepsilon_{\AM},S_{\AM})$ satisfying the standard system of axioms (see \cite{Sweedler}). In what follows we shall use Sweedler notation 
\[\Delta_{\AM_1}(a) = a_{(1)}\otimes a_{(2)}\] when convenient. Using this notation we have for instance 
\begin{equation}\label{basicpr}
\begin{split}
a_{(1)}S_{\AM}(a_{(2)}) &= S_{\AM}(a_{(1)})a_{(2)} = \varepsilon_{\AM}(a)\I_{\AM}\\
a_{(1)} \varepsilon_{\AM}(a_{(2)}) & =  \varepsilon_{\AM}(a_{(1)})a_{(2)} = a
\end{split}
\end{equation}
for all $a\in \AM$.

 Given a $k$-vector space $\mathcal V$, $\mathcal{L}(\mathcal{V})$ denotes the algebra of endomorphism of $\mathcal{V}$. 
For the later needs we formulate the following simple lemma.
\begin{lemma}\label{easy}
Let $\mathcal{H}$ be a Hopf algebra, $\mathcal{V}$   a vector space and $T_1,T_2:\AM\rightarrow\mathcal{V}$ the linear maps such that 
\[(\id_{\AM}\otimes T_1)\circ\Delta_{\AM} = (\id_{\AM}\otimes T_2)\circ\Delta_{\AM}.\] Then $T_1=T_2$. Similarly, \[(T_1\otimes\id_{\AM})\circ\Delta_{\AM} = (T_2\otimes\id_{\AM})\circ\Delta_{\AM}\] implies $T_1=T_2$. 
\end{lemma}
\begin{proof}
For all $a\in\AM$ we have 
\[T_1(a) = (\varepsilon_{\AM}\otimes T_1)\circ\Delta_{\AM}(a) =  (\varepsilon_{\AM}\otimes T_2)\circ\Delta_{\AM}(a) = T_2(a) \] thus $T_1=T_2$.
\end{proof}
The motivation of the following definition was given in the Introduction. 
\begin{definition}\label{qfh}
Let $\BM$ be an algebra, $\mathcal{H}_1 =  (\AM_1,\Delta_{\AM_1},S_{\AM_1}, \varepsilon_{\AM_1})$,  $\mathcal{H}_2= (\AM_2,\Delta_{\AM_2}, ,S_{\AM_2}, \varepsilon_{\AM_2})$  Hopf algebras  and let $\alpha:\AM_1\rightarrow\BM\otimes \AM_2 $ be a morphism. We say that $\alpha$ is a {\it quantum family of homomorphisms} from $\mathcal{H}_1$ to $\mathcal{H}_2$  if 
\begin{equation}\label{eqqfh} (m_{\BM}\otimes\id\otimes\id)\circ(\id\otimes\sigma_{\AM_2\otimes\BM}\otimes\id)\circ(\alpha\otimes\alpha)\circ\Delta_{\AM_1}=(\id\otimes\Delta_{\AM_2})\circ\alpha.\end{equation}
\end{definition}
Let us note that in the Sweedler notation, the condition of Definition \ref{qfh} reads 
\begin{equation}\label{swcond}\alpha(a_{(1)})_{12}\alpha(a_{(2)})_{13} = (\id\otimes\Delta_{\AM_2})\circ\alpha(a)\end{equation}
for all $a\in \AM_1$.
\begin{remark}\label{weakcorem}
A similar concept was considered in the beginning of Paragraph 2.3 of \cite{AD1} in the following context.  Let $(\AM,\Delta_{\AM}, \varepsilon_{\AM})$ be a coalgebra and $(\BM,\Delta_{\BM},S_{\BM}, \varepsilon_{\BM})$ a Hopf algebra. Let $m_{24}:\AM\otimes\BM\otimes\AM\otimes\BM$  be the linear map $c\otimes h \otimes d\otimes k\mapsto c  \otimes d\otimes hk$. A linear map $\rho: \AM\to\AM\otimes \BM$ is called a weak coaction if it satisfies the following conditions 
\begin{equation}\label{weakcoact}
\begin{split}
(\Delta_\AM\otimes\id)\circ\rho &= m_{24}\circ(\rho\otimes\rho)\circ\Delta_\AM \\
(\varepsilon_{\AM}\otimes\id)\circ\rho &= \varepsilon_{\AM}\otimes \I\\
(\id\otimes \varepsilon_{\BM})\circ\rho &= \id_\AM.
\end{split}
\end{equation}
The first from the above conditions has a form which up to the flip of $\BM\otimes \AM$ is the same as the condition \eqref{eqqfh} (with $\AM_1=\AM_2=\AM$). Let us note that from the Hopf algebra structure of $\BM$ only the counit $\varepsilon_{\BM} $ is used in \eqref{weakcoact}. We shall see that the analog of the second condition of \eqref{weakcoact} is automatically satisfied for a quantum family of homomorphisms (see Theorem \ref{counitqfh}).  
\end{remark}
Let $\mathcal{H} = (\AM,\Delta_{\AM},S_{\AM}, \varepsilon_{\AM})$ be a Hopf algebra. 
The algebraic counterpart of the multiplicative unitary  is  the invertible linear map   $W^\mathcal{H}:\AM \otimes \AM \rightarrow\AM \otimes \AM $ such that 
 \[W^\mathcal{H} (a\otimes a') = \Delta_{\AM_1}(a)(\I\otimes a') = a_{(1)}\otimes a_{(2)}a'.\]
The inverse of $W^{\mathcal{H}}$  is  given by  
 \[(W^{\mathcal{H}})^{-1} (a\otimes a') = ((\id\otimes S_{\AM})\Delta_{\AM}(a))(\I\otimes a')=a_{(1)}\otimes S_{\AM}(a_{(2)})a'.\]  
The map  $W\in\mathcal{L}(\AM\otimes\AM)$  yields $W_{12},W_{13},W_{23}\in \mathcal{L}(\mathcal{A} \otimes\mathcal{A} \otimes\mathcal{A})$, e.g. 
  \[W_{12}(a\otimes b\otimes c) = a_{(1)}\otimes a_{(2)}b\otimes c.\] 
 Let us note that 
 \[W^{\mathcal{H}}_{23}W^{\mathcal{H}}_{12}(W^{\mathcal{H}}_{23})^{-1} = W_{12}W_{13}.\] 
Indeed, using the Sweedler notation we get 
 \[W^{\mathcal{H}}_{23}W^{\mathcal{H}}_{12}(a\otimes b\otimes c)  = (a_{(1)}\otimes a_{(2)}\otimes a_{(3)})(\I\otimes b_{(1)}\otimes b_{(2)} c) = W^{\mathcal{H}}_{12}W^{\mathcal{H}}_{13}W^{\mathcal{H}}_{23}(a\otimes b\otimes c)\] 
 for all $a,b,c\in\AM$. 
 
 \begin{theorem}
Let $\mathcal{H}_1  $, $\mathcal{H}_2$ be   Hopf algebras, $\mathcal{B}$  a unital algebra and $\alpha:\AM_1\rightarrow\mathcal{B}\otimes \AM_2$   a morphism.  
Let  us consider a linear map  $U:\AM_1\otimes\mathcal{B}\otimes \AM_2\rightarrow \AM_1\otimes\mathcal{B}\otimes \AM_2$ such that 
\[U(a\otimes b\otimes a') = (\id\otimes\alpha)(\Delta_{\AM_1}(a))(\I\otimes b\otimes a')\]
for all $a\in \AM_1,b\in \BM,a'\in \AM_2$. 
Then $U$ is invertible,  
\[U^{-1}(a\otimes b\otimes a') = (\id\otimes\alpha)((\id\otimes S_{\AM_1})\Delta_{\AM_1}(a))(\I\otimes b\otimes a')\]
and it satisfies 
\begin{equation}\label{simeq}(W^{\mathcal{H}_1}_{12})^{-1}U_{234} W^{\mathcal{H}_1}_{12}  = U_{134} U_{234}.\end{equation}
Moreover, $\alpha$ is a quantum family of homomorphisms from $\mathcal{H}_1$ to $\mathcal{H}_2$ if and only if 
 \begin{equation}\label{alchar}W^{\mathcal{H}_2}_{34}U_{123}(W^{\mathcal{H}_2}_{34})^{-1}  = U_{123}U_{124}.
\end{equation}
\end{theorem}
\begin{proof}
Let  $X:\AM_1\otimes\mathcal{B}\otimes \AM_2\rightarrow \AM_1\otimes\mathcal{B}\otimes \AM_2$ be a linear map  such that
\begin{equation}\label{defX} X(a\otimes b\otimes c) = (\id\otimes\alpha)((\id\otimes S_{\AM_1})\Delta_{\AM_1}(a))(\I\otimes b\otimes c)\end{equation} for all $a\in \AM_1,b\in \BM, c\in \AM_2$. 
Using the Sweedler notation, we compute
\[\begin{split}
UX (a\otimes b\otimes c)&=U((a_{(1)}\otimes \alpha(S_{\AM_1}(a_{(2)})))(\I_{\AM_1}\otimes b\otimes c)\\
&=(a_{(1)}\otimes\alpha(a_{(2)})\alpha(S_{\AM_1}(a_{(3)})))(\I_{\AM_1}\otimes b\otimes c)\\
&=(a_{(1)}\otimes\alpha(a_{(2)}S_{\AM_1}(a_{(3)})))(\I_{\AM_1}\otimes b\otimes c)\\
&=(a_{(1)}\otimes\alpha(\varepsilon (a_{(2)})\I_{\AM_1}))(\I_{\AM_1}\otimes b\otimes c)\\
&=(a_{(1)}\varepsilon (a_{(2)})\otimes b\otimes c\\
&=a\otimes b\otimes c
\end{split}\]
where in the third and the fifth equation we used \eqref{basicpr}.
This shows that   $UX =\id_{\AM_1\otimes\BM\otimes\AM_2}$. Similarly we  check that $XU = \id_{\AM_1\otimes\BM\otimes\AM_2}$.
 
For $a,b\in \AM_1$, $c\in\BM$ and $d\in\AM_2$ we have
\[\begin{split}
U_{234} W^{\mathcal{H}_1}_{12}(a\otimes b\otimes c\otimes d) &=U_{234} (a_{(1)}\otimes a_{(2)}b\otimes c\otimes d)\\
&=(a_{(1)}\otimes a_{(2)}\otimes\alpha(a_{(3)}))(\I\otimes b_{(1)}\otimes\alpha(b_{(2)}))(\I\otimes\I\otimes c\otimes d).
\end{split}
\]
On the other hand  
\[
\begin{split}
W^{\mathcal{H}_1}_{12}U_{134} U_{234}(a\otimes b\otimes c\otimes d) &=W^{\mathcal{H}_1}_{12}U_{134} (a\otimes b_{(1)}\otimes \alpha(b_{(2)}))(\I\otimes\I\otimes c\otimes d)\\&=
W^{\mathcal{H}_1}_{12}(a_{(1)}\otimes \I\otimes \alpha(a_{(2)}))(\I\otimes b_{(1)}\otimes \alpha(b_{(2)}))(\I\otimes \I\otimes c\otimes d)\\&=
(a_{(1)}\otimes a_{(2)}\otimes \alpha(a_{(3)}))(\I\otimes b_{(1)}\otimes \alpha(b_{(2)}))(\I\otimes \I\otimes c\otimes d)
\end{split}
\]
which proves \eqref{simeq}. 

 Let $a\in\AM_1$, $b\in\mathcal{B}$ and $c,d\in\AM_2$. We compute
\[\begin{split}
W^{\mathcal{H}_2}_{34}U_{123}(W^{\mathcal{H}_2}_{34})^{-1}(a\otimes b\otimes c \otimes d)&=W^{\mathcal{H}_2}_{34}U_{123}(a\otimes b\otimes c_{(1)}\otimes S_{\AM_2}(c_{(2)})d)\\
&=(a_{(1)}\otimes((\id\otimes\Delta_{\AM_2})(\alpha(a_{(2)}))))(\I\otimes b\otimes c\otimes d).
\end{split}\]
On the other hand we get
\[\begin{split}
U_{123}U_{124}(a\otimes b\otimes c\otimes d)&= U_{123} (a_{(1)}\otimes \alpha(a_{(2)})_{24})(\I\otimes b\otimes c\otimes d)\\&=(a_{(1)}\otimes\alpha(a_{(2)})_{23}\alpha(a_{(3)})_{24})(\I\otimes b\otimes c\otimes d).
\end{split}\]
If  $\alpha$ is a quantum family of homomorphism, then  we may replace the elements   $(\id\otimes\Delta_{\AM_2})(\alpha(a_{(2)}))$ with $ \alpha(a_{(2)})_{23}\alpha(a_{(3)})_{24}$ which yields  \eqref{alchar}. 

Conversely, if \eqref{alchar} holds then putting $b=\I_{\BM}, c=d=\I_{\AM_2}$ we get 
\[a_{(1)}\otimes\alpha(a_{(2)})_{23}\alpha(a_{(3)})_{24} = a_{(1)}\otimes (\id\otimes\Delta_{\AM_2})(\alpha(a_{(2)})).\] Using 
Lemma \ref{easy} we get $\alpha(a_{(1)})_{12}\alpha(a_{(2)})_{13} = (\id\otimes\Delta_{\AM_2})(\alpha(a))$ for all $a\in \AM_1$, i.e. $\alpha$ is a quantum family of  homomorphisms from $\mathcal{H}_1$ to $\mathcal{H}_2$. 
\end{proof}
 
\begin{theorem}\label{mainthm}
Let $\mathcal{H}_1  $, $\mathcal{H}_2$ be   Hopf algebras  and $\alpha:\AM_1\rightarrow\mathcal{B}\otimes \AM_2$  a quantum family of homomorphisms from  $\mathcal{H}_1$ to $\mathcal{H}_2$.    Then for any $a,a'\in \AM_1$ we have 
\[\alpha(a)_{12}\alpha(a')_{13} = \alpha(a')_{13}\alpha(a)_{12}.\] 
\end{theorem}
\begin{proof}
Using  
 \begin{equation}\label{ident}
\begin{split}W^{\mathcal{H}_2}_{34}U_{123}(W^{\mathcal{H}_2}_{34})^{-1}& = U_{123}U_{124}\\
(W^{\mathcal{H}_1}_{12})^{-1}U_{234} W^{\mathcal{H}_1}_{12}& = U_{134} U_{234}\end{split}
\end{equation}
 we may see that 
 \[(W^{\mathcal{H}_1}_{12})^{-1}W^{\mathcal{H}_2}_{45}U_{234}(W^{\mathcal{H}_2}_{45})^{-1}W^{\mathcal{H}_1}_{12} = (W^{\mathcal{H}_1}_{12})^{-1}U_{234}U_{235}W^{\mathcal{H}_1}_{12} = U_{134}U_{234}U_{135}U_{235}.\]
 Since $(W^{\mathcal{H}_1}_{12})^{-1}W^{\mathcal{H}_2}_{45} =W^{\mathcal{H}_2}_{45}(W^{\mathcal{H}_1}_{12})^{-1} $ we get the same result computing 
 \[W^{\mathcal{H}_2}_{45}(W^{\mathcal{H}_1}_{12})^{-1}U_{234}W^{\mathcal{H}_1}_{12}(W^{\mathcal{H}_2}_{45})^{-1} = W^{\mathcal{H}_2}_{45}U_{134}U_{234}(W^{\mathcal{H}_2}_{45})^{-1} = U_{134}U_{135}U_{234}U_{235}.\]
 Using the invertibility of $U$  we conclude that 
\begin{equation}\label{coalU} U_{234}U_{135} = U_{135}U_{234}.\end{equation}
 Applying \eqref{coalU}  to  a simple tensor $y\otimes x\otimes \I_\BM \otimes \I_{\AM_2}\otimes \I_{\AM_2}$, with $x,y\in \AM_1$ we get  
 \begin{equation}\label{idcom}(\id\otimes\alpha)(\Delta_{\AM_1}(x))_{234}(\id\otimes\alpha)(\Delta_{\AM_1}(y))_{135}=(\id\otimes\alpha)(\Delta_{\AM_1}(y))_{135}(\id\otimes\alpha)(\Delta_{\AM_1}(x))_{234}.\end{equation}
Applying $(\varepsilon\otimes\varepsilon\otimes\id\otimes\id\otimes\id)$ to \eqref{idcom}  we conclude that
 \[\alpha(x)_{12}\alpha(y)_{13} = \alpha(y)_{13}\alpha(x)_{12}\] for all $x,y\in \AM_1$.
\end{proof}

\begin{theorem}\label{counitqfh}
Let $\mathcal{H}_1  $, $\mathcal{H}_2$ be   Hopf algebras and $\alpha:\AM_1\rightarrow\mathcal{B}\otimes \AM_2$  a quantum family of homomorphisms from $\mathcal{H}_1$ to $\mathcal{H}_2$.  Then
\begin{equation}\label{pser}(\id\otimes\varepsilon_{\AM_2})(\alpha(a)) = \varepsilon_{\AM_1}(a)\I_{\BM}\end{equation} for all $a\in \AM_1$.
\end{theorem}
\begin{proof}
The identity 
\[W^{\mathcal{H}_2}_{34}U_{123}(W^{\mathcal{H}_2}_{34})^{-1} = U_{123}U_{124}\] enables us to prove that 
\begin{equation}\label{ep1}(\id_{\AM_1}\otimes\id_{\BM}\otimes\varepsilon_{\AM_2})(U(a\otimes b\otimes c)) = \varepsilon_{\AM_2}(c) a\otimes b\end{equation}
for all $a\in \AM_1$, $b\in\BM$ and $c \in \AM_2$. 
Indeed, let additional $d\in\AM_2$. 
On one hand we have 
\[\begin{split}
W^{\mathcal{H}_2}_{34}U_{123}(W^{\mathcal{H}_2}_{34})^{-1}(a\otimes b\otimes c \otimes d)&=(a_{(1)}\otimes((\id_{\BM}\otimes\Delta_{\AM_2})(\alpha(a_{(2)}))))(\I_{\AM_1}\otimes b\otimes c\otimes d)
\end{split}\]
thus 
\begin{equation}\label{uv}\begin{split}
(\id_{\AM_1}\otimes\id_{\BM}\otimes\id_{\AM_2}\otimes\varepsilon_{\AM_2})&(W^{\mathcal{H}_2}_{34}U_{123}(W^{\mathcal{H}_2}_{34})^{-1}(a\otimes b\otimes c \otimes d))=\\&=
(a_{(1)}\otimes\alpha(a_{(2)}))(\I_{\AM_1}\otimes b\otimes c)\varepsilon_{\AM_2}(d) \\&=
U_{123}(a\otimes b\otimes c)\varepsilon_{\AM_2}(d)
\end{split}\end{equation}
where in the first equality we used the first equation of  \eqref{basicpr}.
On the other hand we have 
\[\begin{split}
(\id\otimes\id\otimes\id\otimes\varepsilon_{\AM_2})&(U_{123}U_{124}(a\otimes b\otimes c\otimes d))=\\&=U_{123} (\id_{\AM_1}\otimes\id_{\BM}\otimes\id_{\AM_2}\otimes\varepsilon_{\AM_2})  (U_{124}(a\otimes b\otimes c\otimes d))
\end{split}\] which together with \eqref{uv} shows that 
\begin{equation}\label{auxU}U_{123} (\id_{\AM_1}\otimes\id_{\BM}\otimes\id_{\AM_2}\otimes\varepsilon_{\AM_2})(U_{124}(a\otimes b\otimes c\otimes d)) = U_{123}(a\otimes b\otimes c)\varepsilon_{\AM_2}(d).\end{equation}
Using the  invertibility of $U$ we cancel $U_{123}$ in \eqref{auxU} and  we get \eqref{ep1}. 
Putting $b = \I_\BM,c = \I_{\AM_2}$ in \eqref{ep1}  we get
\[(\id_{\AM_1}\otimes((\id_{\BM}\otimes\varepsilon_{\AM_2})\circ\alpha))\circ\Delta_{\AM_1}(a)) = a\otimes\I_{\BM} =(\id_{\AM_1}\otimes \varepsilon_{\AM_1}(\cdot)\I_{\BM})\circ \Delta_{\AM_1}(a).\]
Using Lemma \ref{easy} we  conclude that 
\[(\id_{\BM}\otimes\varepsilon_{\AM_2})(\alpha(a)) = \varepsilon_{\AM_1}(a)\I_{\BM}\] for all $a\in\AM_1$. 
\end{proof}

\begin{theorem}\label{Spres}
Let $\mathcal{H}_1  $, $\mathcal{H}_2$ be   Hopf algebras  and $\alpha:\AM_1\rightarrow\mathcal{B}\otimes \AM_2$  a quantum family of homomorphisms  from $\mathcal{H}_1$ to $\mathcal{H}_2$. Then 
\[\alpha\circ S_{\AM_1} = (\id_{\BM}\otimes S_{\AM_2})\circ \alpha.\]
\end{theorem}
\begin{proof}
Let us define an operator $Y\in\mathcal{L}(\AM_1\otimes\BM\otimes\AM_2)$
\begin{equation}\label{defY} Y(a\otimes b\otimes c) = (a_{(1)}\otimes ((\id_{\BM}\otimes S_{\AM_2})\alpha(a_{(2)})))(\I_{\AM_1}\otimes b\otimes c)\end{equation} where $a\in \AM_1,  b\in\BM, c\in\AM_2$. 
We compute 
\[
\begin{split}
U&Y(a\otimes b\otimes c)=U(a_{(1)}\otimes ((\id_{\BM}\otimes S_{\AM_2})\alpha(a_{(2)})))(\I_{\AM_1}\otimes b\otimes c)=\\
&=(a_{(1)}\otimes \alpha(a_{(2)}))(\I_{\AM_1}\otimes (\id_{\BM}\otimes S_{\AM_2})\alpha(a_{(3)}))(\I_{\AM_1}\otimes b\otimes c)\\
& = (\id_{\AM_1}\otimes\id_{\BM} \otimes m_{\AM_2})\left((\id_{\AM_1}\otimes \id_{\BM}\otimes  \id_{\AM_2}\otimes S_{\AM_2})(a_{(1)}\otimes \alpha(a_{(2)})_{23}\alpha(a_{(3)})_{24})\right)(\I_{\AM_1}\otimes b\otimes c)
\\
& = (\id_{\AM_1}\otimes\id_{\BM} \otimes m_{\AM_2})\left((\id_{\AM_1}\otimes \id_{\BM}\otimes \id_{\AM_2}\otimes S_{\AM_2})(a_{(1)}\otimes (\id_{\BM}\otimes\Delta_{\AM_2})(\alpha(a_{(2)})))\right)(\I_{\AM_1}\otimes b\otimes c)
\\
& = (a_{(1)}\otimes  ((\id_{\BM}\otimes\varepsilon_{\AM_2})(\alpha(a_{(2)}))) \otimes \I_{\AM_2})(\I_{\AM_1}\otimes b\otimes c)\\
& = (a_{(1)}\varepsilon_{\AM_1}(a_{(2)})\otimes \I_\BM\otimes \I_{\AM_2}) (\I_{\AM_1}\otimes b\otimes c) = a\otimes b\otimes c
\end{split}
\] where in the fifth equality we used the fact that $\alpha$ is a quantum family of homomorphisms of Hopf algebras; in the sixth equality we used \eqref{basicpr}; in the  seventh equality we used \eqref{pser}. 
Since $U$ is invertible, we get  $Y=U^{-1}$. Putting $b=\I_\BM$ and $c=\I_{\AM_2}$ in \eqref{defX} and \eqref{defY}  we get 
\[(\id_{\AM_1}\otimes\alpha)((\id_{\AM_1}\otimes S_{\AM_1})\Delta_{\AM_1}(a)) = (\id_{\AM_1}\otimes \id_{\BM}\otimes S_{\AM_2})(\id_{\AM_1}\otimes\alpha)(\Delta_{\AM_1}(a)) \] for all $a\in\AM_1$. 
Using Lemma \ref{easy} we get the desired equality. 
\end{proof}
\subsection{Adjoint coaction and cocentralizers}\label{exac}
Let $\mathcal{H} = (\AM,\Delta_{\AM},\varepsilon_{\AM},S_{\AM})$ be a Hopf algebra. The adjoint coaction 
$\cad:\AM\to\AM\otimes \AM$ is a linear map such that 
\[\cad(x) = x_{(1)}S_{\AM}(x_{(3)})\otimes x_{(2)}.\] Note that 
\[(\id\otimes \cad)\circ\cad(x) = x_{(1)}S_{\AM}(x_{(5)})\otimes x_{(2)}S_{\AM}(x_{(4)})\otimes x_{(3)} = (\Delta_\AM\otimes\id)\circ\cad(x).\] 
Moreover 
\[\begin{split}(m\otimes\id\otimes\id)\circ(\id\otimes\sigma_{\AM\otimes\AM}\otimes\id)&\circ(\cad\otimes\cad)\circ\Delta_{\AM}(x) = x_{(1)}S_{\AM}(x_{(3)})x_{(4)}S_{\AM}(x_{(6)})\otimes x_{(2)}\otimes x_{(5)}\\& = x_{(1)}S_{\AM}(x_{(4)})\otimes x_{(2)}\otimes x_{(3)}\\& = (\id\otimes \Delta_{\AM})\circ\cad(x).
\end{split}
\]Thus if $\cad$ is an algebra homomorphism then it can be viewed as a quantum family of homomorphisms from $\AM$ to $\AM$. Using  \cite[Theorem 4.2]{ChKasp} we see that this is the case if and only if $\cad(\AM)\subset \mathcal{Z}(\AM)\otimes \AM$ where $ \mathcal{Z}(\AM)$ is the center of $\AM$. Thus we get the algebraic counterpart of Theorem \ref{examp}. 

In the context of \cite{AD1} the adjoint coaction was used for the construction of a   cocenter of a Hopf algebra (see also \cite{ChKasp}). In what follows we shall introduce and discuss  the concept of cocentralizer of a  given  algebra morphism $\Phi:\AM\to\BM$.

\begin{definition}\label{univcom}
Let $\BM$ and $\CM$ be  algebras over a field $k$ and   $\Phi:\AM\to \BM$,  $\Psi:\AM\to \CM$ surjective homomorhisms.   We say that $\Psi$ and $\Phi$  cocommute if 
\[(\Psi\otimes \Phi)\circ\Delta = (\Psi\otimes \Phi)\circ\Delta^{\textrm{op}}.\]
We say that the surjective algebra homomorophism $\Psi^{\textrm{u}}: \AM\to \CM^{\textrm{u}}$ is the {\it cocentralizer} of  $\Phi$ if $\Psi^{\textrm{u}}$ cocommutes with $\Phi$ and for each $\Psi:\AM\to\CM$ cocommuting with $\Phi$ there exists a surjective homomorphism $\pi:\CM^{\textrm{u}}\to \CM$ such that $\Psi = \pi\circ\Psi^{\textrm{u}}$. 
\end{definition}
\begin{proposition} \label{propcom}  
$\Psi$ and $\Phi$  cocommute if and only if 
\begin{equation}\label{coad}(\Phi\otimes\Psi)(\cad(x)) = \I\otimes\Psi(x)\end{equation} for all $x\in \AM$. 
\end{proposition}
\begin{proof}
Let us consider $\pi_1:\AM\to \BM\otimes \CM$, $\pi_1(x) = \Phi(x)\otimes\I$ and  $\pi_2:\AM\to \BM\otimes \CM$, $\pi_2(x) = \I\otimes\Psi(x)$.  Denoting the convolution of $\pi_1$ and $\pi_2$ by $\pi_1*\pi_2$ we see that    $\Phi$ and $\Psi$ cocommute if and only if $\pi_1*\pi_2 = \pi_2*\pi_1$. Let us note that $\pi_1$ is convolution invertible and its inverse is  $\pi_3(x) = \Phi(S_\AM(x))\otimes\I$. Indeed  
\[
\begin{split}
\pi_1*\pi_3(x) &=  \Phi(x_{(1)})\Phi(S_\AM(x_{(2)}))\otimes\I\\
&= \Phi(x_{(1)}) S_\AM(x_{(2)}))\otimes\I\\&=\varepsilon(x)(\I\otimes\I).
\end{split}
\]
Thus $\Phi$ and $\Psi$  cocommute if and only if 
\[\pi_1*\pi_2*\pi_3 = \pi_2\] which is equivalent with  \eqref{coad}. 
\end{proof}
\begin{remark}
Let us consider $\alpha:\AM\to\BM\otimes\AM$,  
$\alpha = (\Phi\otimes \id)\circ\cad$.  Then $\alpha$ satisfies 
\[ (m_{\BM}\otimes\id\otimes\id)\circ(\id\otimes\sigma_{\AM\otimes\BM}\otimes\id)\circ(\alpha\otimes\alpha)\circ\Delta_{\AM}=(\id\otimes\Delta_{\AM})\circ\alpha.\]
Thus flipping $\BM\otimes\AM$ we can view $\alpha$ as the weak coaction  (see also Remark \ref{weakcorem}). 
Theorem \ref{mainthm} shows that if $\alpha$ is an algebra homomorphism then    $\alpha(x)_{12}\alpha(y)_{13} =\alpha(y)_{13}\alpha(x)_{12} $ for all $x,y\in \AM$. 
\end{remark}
In order to give an explicit description of the cocentralizer of $\Phi$ let us consider 
 $\Psi:\AM\to\CM$ cocommuting  with $\Phi$. Using Proposition \ref{propcom} we see that  
\[\mathcal{I} = \textrm{linspan}\{(\mu\otimes\id)(\alpha(x)) -\mu(\I)x:  x\in\AM,\,\, \mu\in \BM^* \}\subset \ker\Psi.\] Thus defining  an ideal $ \mathcal{J} = \AM\mathcal{I} \AM\subset\AM $ we get the quotient  $\AM\to \AM/\mathcal{J}$  which can be identified with cocentralizer of $\Phi$ in the sens of  Definition \ref{univcom}.  
Using results of \cite[Section 2.3]{AD1} we conclude that $\AM/\mathcal{J}$ can be equipped with the (unique) bialgebra structure such that   the  quotient  $\Psi^{\textrm{u}}:\AM\to\AM/\mathcal{J}$ is  the bialgebra map.

\bibliography{hopf}
\bibliographystyle{plain}
\end{document}